\documentclass[10pt]{amsart}
\usepackage[all,cmtip]{xy}
\usepackage[english]{babel}
\usepackage{graphicx}
\usepackage{amssymb}
\def\today{\number\day\space\ifcase\month\or   January\or February\or
   March\or April\or May\or June\or   July\or August\or September\or
   October\or November\or December\fi\   \number\year}

\theoremstyle{definition}
\newtheorem{thm}{Theorem}[section]
\newtheorem{lemma}[thm]{Lemma}
\newtheorem{prop}[thm]{Proposition}
\newtheorem{df}[thm]{Definition}

\newtheorem{cnj}[thm]{Conjecture}

\newtheorem{rem}[thm]{Remark}

\newtheorem{pbm}[thm]{Problem}

\newcommand{\beq}{\begin{equation}}
\newcommand{\eeq}{\end{equation}}
\newcommand{\beqa}{\begin{eqnarray*}}
\newcommand{\eeqa}{\end{eqnarray*}}
\newcommand{\bal}{\begin{align*}}
\newcommand{\eal}{\end{align*}}
\newcommand{\bi}{\begin{itemize}}
\newcommand{\ei}{\end{itemize}}
\newcommand{\be}{\begin{enumerate}}
\newcommand{\ee}{\end{enumerate}}

\newcommand{\dt}{\delta}
\newcommand{\ep}{\varepsilon}
\newcommand{\zt}{\zeta}

\newcommand{\Z}{{\mathbb{Z}}}
\newcommand{\R}{{\mathbb{R}}}

\newcommand{\N}{{\mathbb{N}}}

\newcommand{\K}{{\mathcal{K}}}

\newcommand{\U}{{\mathcal{U}}}

\newcommand{\T}{{\mathbb{T}}}

\newcommand{\dr}{{\mathrm{dr}}}

\newcommand{\id}{{\mathrm{id}}}

\newcommand{\Aut}{{\mathrm{Aut}}}

\newcommand{\Ad}{{\mathrm{Ad}}}

\newcommand{\dimRok}{{\mathrm{dim}_\mathrm{Rok}}}
\newcommand{\cdimRok}{{\mathrm{dim}_\mathrm{Rok}^\mathrm{c}}}
\newcommand{\dimnuc}{{\mathrm{dim}_\mathrm{nuc}}}

\newcommand{\ifo}{if and only if }

\newcommand{\ca}{$C^*$-algebra}

\newcommand{\uca}{unital $C^*$-algebra}

\newcommand{\hm}{homomorphism}

\newcommand{\Rp}{Rokhlin property}

\newcommand{\Rdim}{Rokhlin dimension}

\newcommand{\I}{\infty}

\title[]{Regularity properties and Rokhlin dimension for compact group actions}

\author{Eusebio Gardella}

\date{\today}

\thanks{This material is based upon work supported by the
  US National Science Foundation through Grant
DMS-1101742; the Johnson Fellowship at the University of Oregon;
and by the the Deutsche Forschungsgemeinschaft
(SFB 878). These sources of financial support are
gratefully acknowledged.}

\address{Department of Mathematics, University  of Oregon,
      Eugene OR 97403-1222, USA, and Fields Institute,
222 College Street, Toronto ON M5T 3J1, Canada.}

\email[]{gardella@uoregon.edu, egardell@fields.utoronto.ca}

\begin{document}

\begin{abstract} We show that formation of crossed products and passage to fixed point algebras by compact
group actions with finite Rokhlin dimension
preserve the following regularity properties: finite decomposition rank, finite nuclear dimension, and
tensorial absorption of the Jiang-Su algebra, the latter in the formulation with commuting towers. \end{abstract}

\maketitle

\tableofcontents

\section{Introduction}

The Elliott conjecture predicts that simple, separable, nuclear \ca s may be classified by their so-called
Elliott invariant, which is essentially $K$-theoretical in nature. Despite the great success that the classification
program enjoyed in its beginnings (see Section 4 of \cite{ET} for a detailed account), the first counterexamples
appeared in the mid to late 1990's, due to R\o rdam (\cite{R}) and Toms (\cite{T}). These examples suggest two
alternatives: either the invariant should be enlarged (to include, for example, the Cuntz semigroup), or the
class of \ca s should be restricted, assuming further regularity properties (stronger than nuclearity).
Significant effort has been put into both directions, and the present paper is a contribution to the second one of these (in particular, to the verification of certain regularity properties for specific crossed product \ca s).\\
\indent The regularity properties that have been studied are of very different nature: topological, analytical
and algebraic. These are: finite nuclear dimension (or finite decomposition rank, in the stably finite case);
tensorial absorption of the Jiang-Su algebra; and strict comparison of positive elements. These notions are
surveyed in \cite{ET}. \\
\indent Despite their seemingly different flavors, Toms and Winter conjectured these notions to be equivalent
for all unital, nuclear, separable, non-elementary, simple \ca s. Some implications hold in full generality, as
was shown by R\o rdam (\cite{Rordam}) and Winter (\cite{W1} and \cite{W2})), and several partial results are
available for the remaining implications (\cite{MS}, \cite{KR}, \cite{S}, and \cite{TWW}). More
recently, Sato, White and Winter showed that the Toms-Winter conjecture is true if one moreover
assumes that the \ca\ in question
has at most one trace (Corollary C in \cite{SWW}). It should also be pointed out that all three regularity
properties are satisfied by every \ca\ in any of the classes considered by the existing classification
theorems.\\
\indent In view of their importance in the classification program, it is useful to know what constructions
preserve these regularity properties. In this paper, we show that formation of crossed products and passage
to fixed point algebras preserve finiteness of nuclear dimension (Theorem \ref{nuc dim}), finiteness of
decomposition rank (Theorem \ref{dr}), and
tensorial absorption of the Jiang-Su algebra (Theorem \ref{thm: cp Z-abs}), provided that the action has finite Rokhlin dimension in the
sense of \cite{gardella Rdim} (for Jiang-Su absorption, one needs to assume the formulation with commuting
towers). Our work generalizes results for finite groups of Hirshberg, Winter and Zacharias from \cite{HWZ},
where they also studied similar questions for crossed products by automorphisms.\\
\ \\
\indent \textbf{Acknowledgements:} Most of this work was done while the author was visiting the Westf\"alische
Wilhelms-Universit\"at M\"unster in the summer of 2013, and while the author was participating in the Thematic Program
on Abstract Harmonic Analysis, Banach and Operator Algebras, at the Fields Institute for Research in Mathematical Sciences
at the University of Toronto, in January-June 2014. He wishes to thank both Mathematics departments for their
hospitality and for providing a stimulating research environment. \\
\indent The author would like to express his gratitude to Luis Santiago for a valuable conversation regarding Lemma \ref{lma: asymptotic order zero}, as well as Ilan Hirshberg and Chris Phillips for helpful discussions and electronic correspondence. Finally, he would also like to thank Wilhelm Winter for suggesting this line of research.

\section{Preliminaries and notation}

\indent Homomorphisms of \ca s will always be assumed to be $\ast$-homomorphisms. For a \ca\ $A$, we denote by $\Aut(A)$
the automorphism group of $A$. If $G$ is a locally compact group, an action of $G$ on $A$ is a \emph{continuous}
group homomorphism from $G\to \Aut(A)$, unless otherwise stated. If $\alpha\colon G\to\Aut(A)$ is an
action of $G$ on $A$, we will denote by $A^\alpha$ the fixed-point subalgebra of $A$. \\
\indent We take $\N=\{1,2,\ldots\}$, and for $n$ in $\N$, we write $\Z_n$ for the cyclic group
$\Z/ n\Z$.\\
\indent All groups will be second countable. By a theorem of Birkoff-Kakutani (Theorem 1.22 in \cite{montgomery zippin}), a topological
group is metrizable if and only if it is first countable. In particular, all our groups will be metrizable. It is well-known that a compact metrizable group admits a translation-invariant metric. We will implicitly choose such a metric on all
our groups, which will be denoted by $d$.\\

\subsection{Central sequence algebras} Let $A$ be a \uca, and denote by $\ell^\I(\N,A)$ the \ca\ of all bounded
sequences $(a_n)_{n\in\N}$ in $A$ with the supremum norm
and pointwise operations. The set
$$c_0(\N,A)=\left\{(a_n)_{n\in\N}\in\ell^\I(\N,A)\colon \lim_{n\to\I}\|a_n\|=0\right\}.$$
is an ideal in $\ell^\I(\N,A)$, and we denote the corresponding quotient
by $A_\I$. Write $\kappa_A\colon \ell^\I(\N,A)\to A_\I$ for the quotient map, and identify $A$ with the
unital subalgebra of $\ell^\I(\N,A)$ consisting of the constant sequences, and with a unital subalgebra of
$A_\I$ by taking its image under $\kappa_A$. We write $A_\I\cap A'$ for the relative commutant of $A$
inside of $A_\I$. \\
\indent If $\alpha\colon G\to\Aut(A)$ is an action of $G$ on $A$, there are (not necessarily continuous) actions
of $G$ on $A_\I$
and on $A_\I\cap A'$, both denoted by $\alpha_\I$. We set
$$\ell^\I_\alpha(\N,A)=\{a\in \ell^\I(\N,A)\colon g\mapsto (\alpha_\I)_g(a) \ \mbox{ is continuous}\},$$
and $A_{\I,\alpha}=\kappa_A(\ell^\I_\alpha(\N,A))$. By construction, $A_{\I,\alpha}$ is invariant under $\alpha_\I$,
and the restriction of $\alpha_\I$ to $A_{\I,\alpha}$ is continuous.\\
\ \\
\indent The following proposition relates the crossed product functor with the sequence algebra
functor.

\begin{prop} \label{prop: cp and sequence algebra commute}
Let $A$ be a \uca, let $G$ be a compact group and let $\alpha\colon G\to\Aut(A)$ be an action. Then there is a canonical
embedding
\[A_{\I,\alpha}\rtimes_{\alpha_\I} G\hookrightarrow (A\rtimes_\alpha G)_\I.\]
\end{prop}
\begin{proof}
Note that if $B$ is a \ca, then there is a unital map $M(B)_\I\to M(B_\I)$. The canonical maps $A\to M(A\rtimes_\alpha G)$
and $G\to M(A\rtimes_\alpha G)$ induce canonical maps
\begin{align*} A_{\I,\alpha}&\to (M(A\rtimes_\alpha G))_\I\to M((A\rtimes_\alpha G)_\I) \ \ \mbox{and} \\
G&\to (M(A\rtimes_\alpha G))_\I\to M(A\rtimes_\alpha G)_\I\end{align*}
which satisfy the covariance condition for $\alpha_\I$. It follows from the universal property of the crossed product
$A_{\I,\alpha}\rtimes_{\alpha_\I} G$ that there is a map as in the statement, which is injective because so is
$A_{\I,\alpha}\to (M(A\rtimes_\alpha G))_\I$.\end{proof}

In the proposition above, the canonical embedding will in general not be surjective unless $G$ is finite.

\subsection{Order zero maps and Rokhlin dimension for compact group actions}
We briefly recall some of the basics of completely positive
order zero maps. The reader is referred to \cite{winter zacharias} for more details and further results.\\
\ \\
\indent Let $A$ be a \ca. We say that two elements $a$ and $b$ in $A$ are \emph{orthogonal},
and write $a\perp b$, if $ab=ba=a^*b=ab^*=0$. If $a,b\in A$ are selfadjoint, then they are orthogonal \ifo
$ab=0$.

\begin{df} Let $A$ and $B$ be \ca s, and let $\varphi\colon A\to B$ be a completely positive map. We say
that $\varphi$ has \emph{order zero} if $\varphi$ preserves orthogonality.\end{df}

It is straightforward to check that $C^*$-algebra \hm s have order zero, and that the composition of two
order zero maps is again order zero.

It is a well-known fact that equivariant homomorphisms between dynamical systems induce homomorphisms between the respective crossed products,
a fact that can be easily seen by considering the universal property of such objects. Using the structure of order zero maps, it follows that
an analogous statement holds for completely contractive order zero maps that are equivariant.

\begin{prop} \label{prop: order zero equiva cp}
Let $A$ and $B$ be \ca s, let $G$ be a locally compact group, let $\alpha\colon G\to\Aut(A)$ and $\beta\colon G\to\Aut(B)$ be
actions, and let $\rho\colon A\to B$ be an equivariant completely positive contractive order zero map. Then $\rho$ induces a canonical
completely positive contractive order zero map
\[\sigma\colon A\rtimes_\alpha G\to B\rtimes_\beta G.\]
\end{prop}
\begin{proof}
Denote by $\varphi_\rho\colon C_0((0,1])\otimes A\to B$ be the homomorphism determined by $\rho$ as in Corollary 3.1 in \cite{winter zacharias}.
Denote by $\widetilde{\alpha}$
the diagonal action $\widetilde{\alpha}=\id_{C_0((0,1])}\otimes\alpha$ of $G$ on $C_0((0,1])\otimes A$. Since $\varphi_\rho$ is equivariant
with respect to this action by Corollary 2.10 in \cite{gardella Rdim}, there is a canonical homomorphism
$$\psi\colon (C_0((0,1])\otimes A)\rtimes_{\widetilde{\alpha}}G\to B\rtimes_\beta G.$$
Identify $(C_0((0,1])\otimes A)\rtimes_{\widetilde{\alpha}}G$ with $C_0((0,1])\otimes (A\rtimes_\alpha G)$ in the usual way. Then the order zero map
$\sigma\colon A\rtimes_\alpha G\to B\rtimes_\beta G$ given by $\sigma(x)=\psi(\id_{(0,1]}\otimes x)$ for $x$ in $A\rtimes_\alpha G$, is the desired order zero map. \end{proof}

Finally, we recall the definition of Rokhlin dimension for a compact group action from
\cite{gardella Rdim}.
Given a compact group $G$, we denote by $\verb'Lt'\colon G\to\Aut(C(G))$ the
action of left translation.

\begin{df}\label{def finite Rdim}
Let $G$ be a second countable compact group, let $A$ be a \uca,
and let $\alpha\colon G\to\Aut(A)$ be a continuous action. Given a nonnegative integer $d$, we 
say that $\alpha$ has \emph{\Rdim\ $d$}, and write $\dimRok(\alpha)=d$, if
$d$ is the least integer such that there exist equivariant completely positive contractive order zero
maps
$$\varphi_0,\ldots,\varphi_d\colon (C(G),\texttt{Lt})\to (A_{\I,\alpha}\cap A',\alpha_\I)$$
such that $\varphi_0(1)+\ldots+\varphi_d(1)=1$. If no integer $d$ as above exists, we say that
$\alpha$ has \emph{infinite Rokhlin dimension}, and denote it by $\dimRok(\alpha)=\I$.\\
\indent Finally, if one can always choose the maps $\varphi_0,\ldots,\varphi_d$ to have commuting ranges,
then we say that $\alpha$ has \emph{Rokhlin dimension $d$ with commuting towers}, and denote it by
$\cdimRok(\alpha)=d$.\end{df}

\section{Preservation of finite nuclear dimension and decomposition rank}

In this section, we explore the structure of the crossed product and fixed point algebra of an
action of a compact group with finite Rokhlin dimension in relation to their nuclear dimension and
decomposition rank. Specifically, we show that finite nuclear
dimension and finite decomposition rank are inherited by the crossed product and fixed point algebra by any
such action. \\
\ \\
\indent We introduce some notation that will be used in this subsection.\\
\indent Let $A$ be a \ca, let $G$ be a compact group, and let $\alpha\colon G\to\Aut(A)$ be a continuous action. Give $C(G)\otimes A$ the diagonal
action $\verb'Lt'\otimes\alpha$ of $G$. Then the canonical inclusion $A\to A \otimes C(G)$ is equivariant. Identify $A\otimes C(G)$ with
$C(G,A)$ in the usual way, and let
$$\theta\colon (C(G,A),\texttt{Lt}\otimes\alpha)\to (C(G,A),\texttt{Lt}\otimes\id_A)$$
be given by $\theta(\xi)(g)=\alpha_{g^{-1}}(\xi(g))$ for all $\xi$ in $C(G,A)$ and all $g$ in $G$. Then $\theta$ is clearly an isomorphism,
and it is moreover equivariant, since
\begin{align*} \theta\left((\texttt{Lt}_g\otimes\alpha_g)(\xi)\right)(h)&=\alpha_{h^{-1}}\left((\texttt{Lt}_g\otimes\alpha_g)(\xi)(h)\right)\\
&=\alpha_{h^{-1}}\left(\alpha_g(\xi(g^{-1}h))\right)\\
&=\theta(\xi)(g^{-1}h)\\
&= (\texttt{Lt}_g\otimes\id_A)(\xi)(h)
\end{align*}
for all $g$ and $h$ in $G$, and all $\xi$ in $C(G,A)$. It follows that there are isomorphisms
$$(A\otimes C(G))\rtimes_{\texttt{Lt}\otimes\alpha}G\cong A\otimes (C(G)\rtimes_{\texttt{Lt}}G)\cong A\otimes \K(L^2(G)).$$
We conclude that the canonical inclusion $A\to A\otimes C(G)$ induces an injective homomorphism
$$\iota\colon A\rtimes_\alpha G\to A\otimes \K(L^2(G)),$$
which we will refer to as the \emph{canonical} embedding of $A\rtimes_\alpha G$ into $A\otimes\K(L^2(G))$. This terminology is justified
by the following observation.

\begin{rem} Adopt the notation from the discussion above, and denote by $\lambda\colon G\to \U(L^2(G))$ the left regular representation, and identify $A\rtimes_\alpha G$ with its image
under $\iota$. Then
$$A\rtimes_\alpha G= \left(A\otimes\K(L^2(G))\right)^{\alpha\otimes\Ad(\lambda)}.$$\end{rem}

The following lemma will be the main technical device we will use to prove Theorem \ref{dr}. When the group $G$ is finite
(as is considered in \cite{HWZ}), the desired ``almost'' order zero maps are constructed using the elements of appropriately chosen towers
in the definition of finite Rokhlin dimension. In the case of an arbitrary compact group, some work is needed to get such maps.\\
\indent Lemma \ref{lma: asymptotic order zero} can be thought of as an analog of Osaka-Phillips' approximation of crossed products of actions with the
Rokhlin property by matrices over corners of the underlying algebra, which was used in \cite{osaka phillips}. This approximation technique
is implicit in the paper \cite{hirshberg winter}, and further applications of it will appear in \cite{gardella hirshberg santiago}. We
would like to thank Luis Santiago for suggesting such an approach.

\begin{lemma} \label{lma: asymptotic order zero}
Let $A$ be a unital, nuclear \ca, let $G$ be a compact group, let $d$ be a nonnegative integer, and let $\alpha\colon G\to\Aut(A)$ be an
action with $\dimRok(\alpha)\leq d$. Denote by $\iota\colon A\rtimes_\alpha G\to A\otimes \K(L^2(G))$ the canonical embedding. \\
\indent Given compact sets $F\subseteq A\rtimes_\alpha G$ and $S\subseteq A\otimes\K(L^2(G))$, and given $\ep>0$,
there are completely positive maps
$$\rho_0,\ldots,\rho_d\colon A\otimes\K(L^2(G))\to A\rtimes_\alpha G$$
such that
\be\item $\left\|\rho_j(a)\rho_j(b)\right\|<\ep$ whenever $a$ and $b$ are positive elements in $S$ with $ab=0$;
\item $\left\|\sum\limits_{j=0}^d(\rho_j\circ\iota)(x)-x\right\|<\ep$ for all $x$ in $F$;
\item The map
$$\sum\limits_{j=0}^d \rho_j\colon \bigoplus\limits_{j=0}^d A\otimes\K(L^2(G))\to A\rtimes_\alpha G$$
is completely positive and contractive. \ee
In other words, the maps $\iota$ and $\rho_0,\ldots,\rho_d$ induce a diagram
\begin{align*} \xymatrix{ A\rtimes_\alpha G\ar[rd]_-\iota \ar[rr]^-{\id_{A\rtimes_\alpha G}} & & A\rtimes_\alpha G\\
 & A\otimes\K(L^2(G))\ar[ur]_-{\sum\limits_{j=0}^d\rho_j}}
\end{align*}
that approximately commutes on $F$ up to $\ep$, and such that the completely positive contractive maps $\rho_j$ are ``almost'' order zero on $S$.
\end{lemma}
\begin{proof}
Let $\varphi_0,\ldots,\varphi_d\colon C(G)\to A_{\I,\alpha}\cap A'$ be the equivariant completely positive contractive order zero maps as in
the definition of Rokhlin dimension at most $d$ for $\alpha$. Upon tensoring with $\id_A$, we obtain equivariant completely positive contractive
order zero maps
$$\psi_0,\ldots,\psi_d\colon A\otimes C(G)\to A_{\I,\alpha},$$
which satisfy $\sum\limits_{j=0}^d \psi_j(a\otimes 1)=a$ for all $a$ in $A$. (The action on $A\otimes C(G)$ is the diagonal, using translation on
$C(G)$.) With $e\in \K(L^2(G))$ denoting the projection onto the constant functions on $G$, use Proposition \ref{prop: order zero equiva cp} and
Proposition \ref{prop: cp and sequence algebra commute} to obtain completely positive contractive order zero maps
$$\sigma_0,\ldots,\sigma_d\colon A\otimes\K(L^2(G))\to (A\rtimes_\alpha G)_\I$$
which satisfy $\sum\limits_{j=0}^d \sigma_j(x\otimes e)=x$ for all $x$ in $A\rtimes_\alpha G$, and such that $\sum\limits_{j=0}^d \sigma_j$ is
contractive.\\
\indent For $j=0,\ldots,d$, use nuclearity of $A$, together with Choi-Effros lifting theorem, to lift $\sigma_j$ to a completely positive contractive
map
$$\rho_j\colon A\otimes\K(L^2(G))\to A\rtimes_\alpha G,$$
which satisfies conditions (1) and (2) of the statement with $\frac{\ep}{2}$ in place of $\ep$, and such that
$$\left\|\sum\limits_{j=0}^d \rho_j\right\|< 1+\frac{\ep}{2}.$$
Dividing each of the maps $\rho_j$ by the above norm introduces an additional error of $\frac{\ep}{2}$, and the resulting rescaled maps are the
desired order zero maps.
\end{proof}

With the aid of Lemma \ref{lma: asymptotic order zero}, the proof of the following theorem can be proved using ideas similar to the ones
used to prove Theorem 1.3 in \cite{HWZ}.

\begin{thm}\label{dr} Let $A$ be a \uca, let $G$ be a compact, and let
$\alpha\colon G\to\Aut(A)$ be a continuous action with finite Rokhlin dimension. Then
$$\dr(A^\alpha)\leq \dr(A\rtimes_\alpha G)\leq (\dimRok(\alpha)+1)(\dr(A)+1)-1.$$\end{thm}
\begin{proof}
The first inequality is a consequence of the fact that $A^\alpha$ is isomorphic to a corner of $A\rtimes_\alpha G$ by compactness of $G$
(see the Theorem in \cite{Rosenberg}), together with Proposition 3.8 in \cite{kirchberg winter}. \\
\indent In order to show the second inequality, it is enough to do it when $\dr(A)<\I$ and $\dimRok(\alpha)<\I$. Set $N=\dr(A)$ and set
$d=\dimRok(\alpha)$. \\
\indent Let $\ep>0$ and let
$F$ be a compact subset of $A\rtimes_\alpha G$.
Choose finite dimensional \ca s $\mathcal{F}_0,\ldots,\mathcal{F}_N$,
a completely positive contractive map
$$\psi\colon A\to \mathcal{F}=\mathcal{F}_0\oplus \cdots\oplus \mathcal{F}_N,$$
and completely positive contractive order zero maps $\phi_\ell\colon \mathcal{F}_\ell\to A$
for $\ell=0,\ldots,N$, such that $\phi=\phi_0+\ldots+\phi_N\colon \mathcal{F}\to A$ is completely positive and
contractive, and satisfies
$$\|(\phi\circ\psi)(a)-a\|<\frac{\ep}{2}$$
for all $a$ in $F$. Denote by $\iota\colon A\rtimes_\alpha G\to A\otimes\K(L^2(G))$ the canonical inclusion. 
We will construct completely positive approximations for $A\rtimes_\alpha G$ of the form
\begin{align*}\xymatrix{ A\rtimes_\alpha G\ar[rrrr]^-{\id_{A\rtimes_\alpha G}}\ar[rd]_{\iota} & & & & A\rtimes_\alpha G,\\
 & A\otimes\K(L^2(G))\ar[rd]_{\bigoplus\limits_{j=0}^d\psi} & & \bigoplus\limits_{j=0}^d A\otimes\K(L^2(G))\ar[ur]_{\rho} & \\
& & \bigoplus\limits_{j=0}^d \mathcal{F}\ar[ur]_{\bigoplus\limits_{j=0}^d\phi} & & }
\end{align*}
where the map $\rho\colon \bigoplus\limits_{j=0}^dA\otimes \K(L^2(G))\to A\rtimes_\alpha G$ will be constructed
later using that $\dimRok(\alpha)\leq d$, in such a way that
$\rho\circ\phi\colon \mathcal{F}\to A\rtimes_\alpha G$ is the sum of ``almost'' order zero maps. We will then use projectivity
of the cone over finite dimensional \ca s to replace the map $\rho\circ\phi$ with maps that are
decomposable into completely positive contractive order zero maps.  \\
\indent Set
$$\ep_1=\frac{\ep}{8(d+1)(N+1)}.$$
Using stability of order zero maps from finite dimensional \ca s, choose $\dt>0$ such that
whenever $\sigma\colon \mathcal{F}\to A\rtimes_\alpha G$ is a completely positive contractive map satisfying
$$\|\sigma(x)\sigma(y)\|< \dt$$
for all positive orthogonal contractions $x$ and $y$ in $\mathcal{F}$, there exists a completely positive contractive
order zero map $\sigma'\colon \mathcal{F}\to A\rtimes_\alpha G$ with $\|\sigma'-\sigma\|<\ep_1$. \\
\indent Let $B_{\mathcal{F}}$ denote the unit ball of $\mathcal{F}$, and set
$$S=\bigcup\limits_{g\in G}\bigcup\limits_{\ell=0}^N\left(\alpha_g\otimes\id_{\K(L^2(G))}\right)(\phi_\ell(B_{\mathcal{F}})),$$
which is a compact subset of $A\otimes\K(L^2(G))$.\\
\indent Set $\ep_2=\min\left\{\dt,\frac{\ep}{4}\right\}$. Use Lemma \ref{lma: asymptotic order zero} to find completely positive contractive maps
$$\rho_0,\ldots,\rho_d\colon A\otimes\K(L^2(G))\to A\rtimes_\alpha G$$
such that
\be\item $\left\|\rho_j(a)\rho_j(b)\right\|<\ep_2$ whenever $a$ and $b$ are positive elements in $S$ with $ab=0$;
\item $\left\|\sum\limits_{j=0}^d(\rho_j\circ\iota)(x)-x\right\|<\ep_2$ for all $x$ in $F$;
\item The map
$$\sum\limits_{j=0}^d \rho_j\colon \bigoplus\limits_{j=0}^d A\otimes\K(L^2(G))\to A\rtimes_\alpha G$$
is completely positive and contractive. \ee

\indent Fix indices $\ell$ in $\{0,\ldots,N\}$ and $j$ in $\{0,\ldots,d\}$, and fix positive orthogonal
elements $x$ and $y$ in $S$. Set $a=\phi_\ell(x)$ and $b=\phi_\ell(y)$. Since $\phi_\ell$ is order zero, we have $ab=0$. Then
$$\|(\rho_j\circ\phi_\ell)(x)(\rho_j\circ\phi_\ell)(y)\|=\|\rho_j(a)\rho_j(b)\|<\ep_2.$$
By the choice of $\dt$, there are completely positive contractive order zero maps
$\sigma_{j,\ell}\colon \mathcal{F}\to A\rtimes_\alpha G$ satisfying
$$\|\sigma_{j,\ell}-\rho_j\circ\phi_\ell\|<\ep_1$$
for $j=0,\ldots,d$ and $\ell=0,\ldots,N$. For $j=0,\ldots,d$, define a linear map $\sigma_j\colon
\mathcal{F}\to A\rtimes_\alpha G$ by
$$\sigma_j=\sum\limits_{\ell=0}^N \sigma_{j,\ell},$$
and let $\sigma\colon\bigoplus\limits_{j=0}^d \mathcal{F}\to A\rtimes_\alpha G$ be given by $\sigma=\sum\limits_{j=0}^d\sigma_j$.
Then $\sigma$ is completely positive, and moreover
$$\left\|\sigma \right\|<1+(d+1)(N+1)\ep_1.$$
Set $\tau=\frac{\sigma}{\|\sigma\|}$, which is completely positive contractive and order zero, and satisfies
$$\left\|\tau-\rho\circ\left(\sum_{j=0}^d\phi\right)\right\|\leq \|\tau-\sigma\|+\left\|\sigma-\rho\circ\left(\sum_{j=0}^d\phi\right)\right\|<2(d+1)(N+1)\ep_1.$$
Finally, we claim that
\begin{align*} \xymatrix{ A\rtimes_\alpha G\ar[dr]_{\bigoplus\limits_{j=0}^d\psi\circ\iota}\ar[rr]^{\id} & & A\rtimes_\alpha G\\
 & \bigoplus_{\ell=0}^d\mathcal{F}\ar[ur]_\tau & }
\end{align*}
approximately commutes on the set $F$ within $\ep$, and that $\tau$ can be decomposed into $(d+1)(N+1)-1$ order zero summands. The only thing that
remains to be checked is that $\|(\tau\circ\psi)(a)-a\|<\ep$ for all $a$ in $F$. Given $a$ in $F$, we estimate as follows:
\begin{align*} \left(\tau\circ\left(\bigoplus_{j=0}^d\psi\circ\iota\right)\right)(a)& \approx_{2(d+1)(N+1)\ep_1}
\left(\rho\circ\left(\sum_{j=0}^d\phi\right)\circ\left(\bigoplus_{j=0}^d\psi\circ\iota\right)\right)(a)\\
&\approx_{\frac{\ep}{2}} (\rho\circ\iota)(a)\\
&\approx_{\ep_2} a.\end{align*}
Hence
\begin{align*} \left\|\left(\tau\circ\left(\bigoplus_{j=0}^d\psi\circ\iota\right)\right)(a)-a\right\|&<2(d+1)(N+1)\ep_1+\frac{\ep}{2}+\ep_2\\
&<\frac{\ep}{4}+ \frac{\ep}{2}+\frac{\ep}{4}=\ep,\end{align*}
and the claim is proved. This finishes the proof.
\end{proof}

The corresponding statement for nuclear dimension is true. Its proof is analogous to that of Theorem
\ref{dr} and is therefore omitted. The difference is that one does not need to take care of the norms of
the components of the approximations.

\begin{thm}\label{nuc dim} Let $A$ be a \uca, let $G$ be a compact group, and let
$\alpha\colon G\to\Aut(A)$ be a continuous action with finite Rokhlin dimension. Then
$$\dimnuc(A^\alpha)\leq \dimnuc(A\rtimes_\alpha G)\leq (\dimRok(\alpha)+1)(\dimnuc(A)+1)-1.$$\end{thm}

\begin{rem} It should be pointed out that the inequalities $\dr(A^\alpha)\leq \dr(A\rtimes_\alpha G)$ and
$\dimnuc(A^\alpha)\leq \dimnuc(A\rtimes_\alpha G)$ are likely to be equalities whenever $\alpha$ has
finite Rokhlin dimension. Indeed, it is probably the case, although we have not checked, that finite
Rokhlin dimension implies saturation (see Definition 5.2 in \cite{phillips survey}),
from which it would follow that $A^\alpha$ and $A\rtimes_\alpha G$ are Morita equivalent, and hence they
have the same nuclear dimension and decomposition rank. We remark that saturation is automatic whenever
the crossed product is simple, so the equalities $\dr(A^\alpha)= \dr(A\rtimes_\alpha G)$ and
$\dimnuc(A^\alpha)= \dimnuc(A\rtimes_\alpha G)$ hold in many cases of interest. In particular, this is
the case whenever $G$ is finite and $A$ is simple by Theorem 4.14 in \cite{gardella Rdim}.\end{rem}

An example in which one can apply these results to deduce finite nuclear dimension and decomposition
rank of the crossed product is that of free actions of compact Lie groups on compact metric spaces with finite
covering dimension. Indeed, such actions have finite Rokhlin dimension by Theorem 4.2 in \cite{gardella Rdim},
and since $A=C(X)$ has finite nuclear dimension and decomposition rank, we deduce that $A\rtimes G$ does as well.
(Note that since the action is free, Situation 2 in \cite{Rieffel} implies that $A^G$ and $A\rtimes G$
are Morita equivalent.) \\
\indent Nevertheless, there is a much simpler proof of this fact, which even yields a better estimate
of the nuclear dimension and decomposition rank. Indeed, if $X$ is a compact free $G$-space, then the fixed
point algebra of $C(X)$ is $C(X/G)$. Moreover, the orbit space $X/G$ has covering dimension at most
$\dim(X)-\dim (G)$, and hence $C(X/G)$ has finite nuclear dimension and decomposition rank (and equal to each other).
We conclude that
$$\dimnuc(C(X)\rtimes G)=\dr(C(X)\rtimes G)\leq \dim(X)-\dim (G).$$

\section{Preservation of $\mathcal{Z}$-absorption}

We now turn to preservation of $\mathcal{Z}$-absorption, under the stronger assumption that the action have finite
Rokhlin dimension with commuting towers. We will need a technical lemma characterizing $\mathcal{Z}$-absoprtion in
a form that is useful in our context.

\begin{lemma}\label{lma: Z-abs}
Let $A$ be a unital separable \ca, let $G$ be a compact group, and let $\alpha\colon G\to\Aut(A)$
 be a continuous action. Let $d$ be a non-negative integer, and suppose that for any $r$ in $\N$,
for any compact subset $F\subseteq A$, and for any $\ep>0$, there exist
completely positive contractive maps
$$\theta_0,\ldots,\theta_d\colon M_r\to A_{\I,\alpha} \ \ \mbox{ and } \ \ \eta_0,\ldots,\eta_d\colon M_{r+1}\to A_{\I,\alpha}$$
such that the following properties hold for all $x,x'$ in $M_r$ and for all $y,y'$ in $M_{r+1}$ with $\|x\|,\|x'\|,\|y\|,\|y'\|\leq 1$,
for all $g$ in $G$, for all $a$ in $F$ and for all $j,k=0,\ldots,d$:
\beq\left\|\left[\theta_j(x),\eta_k(y)\right]\right\|<\ep ; \eeq
\beq \mbox{if } \ x,x',y,y'\geq 0 \ \mbox{ and } \ x\perp x' \ , \ y\perp y', \mbox{ then }\eeq
\beqa\left\|\theta_j(x)\theta_j(x')\right\|<\ep \ \ \mbox{ and } \ \ \left\|\eta_k(y)\eta_k(y')\right\|<\ep; \eeqa
\beq\|(\alpha_\I)_g(\theta_k(x))-\theta_k(x)\|<\ep \ \ \mbox{ and } \ \ \|(\alpha_\I)_g(\eta_k(y))-\eta_k(y)\|<\ep;\eeq
\beq\|a\theta_k(x)-\theta_k(x)a\|<\ep \ \ \mbox{ and } \ \ \|a\eta_k(y)-\eta_k(y)a\|<\ep;\eeq
\beq\left\|\sum\limits_{k=0}^d\theta_k(1)+\eta_k(1)-1\right\|<\ep.\eeq
Then $A\rtimes_\alpha G$ is $\mathcal{Z}$-stable.
\end{lemma}
\begin{proof} Using stability of completely positive order zero maps from matrix algebras, we may assume
that the maps $\theta_0,\ldots,\theta_d$ and $\eta_0,\ldots,\eta_d$ can always be chosen to satisfy condition
(2) exactly.\\
\indent Let $r$ in $\N$. We claim that there are order zero maps maps
$$\theta_0,\ldots,\theta_d\colon M_r\to (A_{\I}\cap A')^{\alpha_\I} \ \ \mbox{ and }
\ \ \eta_0,\ldots,\eta_d\colon M_{r+1}\to (A_{\I}\cap A')^{\alpha_\I}$$
with $\sum\limits_{k=0}^d \theta_k(1)+\eta_k(1)=1$. Once we prove the claim, the rest of the proof goes exactly as in
Lemma 5.7 in \cite{HWZ}. (There, the authors assumed the group $G$ to be discrete, but since the order zero maps we will produce
land in the fixed point algebra of $A_\I\cap A'$, and in particular, in $A_{\I,\alpha}\cap A'$, the fact that $G$
is not discrete in this lemma is not an issue.)\\
\indent Choose an increasing sequence $(F_n)_{n\in\N}$ of compact subsets of $A$ such that $A_0=\bigcup\limits_{n\in\N}F_n$
is dense in $A$. Without loss of generality, we may assume that $A_0$ is closed under multiplication, addition and
involution. For each $n$ in $\N$, let
$$\theta^{(n)}_0,\ldots,\theta^{(n)}_d\colon M_r\to A_{\I,\alpha} \ \ \mbox{ and }
\ \ \eta^{(n)}_0,\ldots,\eta^{(n)}_d\colon M_{r+1}\to A_{\I,\alpha}$$
be completely positive contractive order zero maps satisfying conditions (1) and (3)-(5) in the statement for $F_n$
and $\ep=\frac{1}{n}$. For each $n$ in $\N$ and each $j=0,\ldots,d$, let
$$\widetilde{\theta}_j^{(n)}\colon M_r\to \ell_{\alpha}^\I(\N,A) \ \ \mbox{ and } \ \
\widetilde{\eta}_j^{(n)}\colon M_{r+1}\to \ell_{\alpha}^\I(\N,A)$$
be completely positive contractive lifts of $\theta_j^{(n)}$ and $\eta_j^{(n)}$ respectively. As in the proof of Lemma 2.4 in
\cite{hirshberg winter}, we can find a strictly increasing sequence $n_k$ of natural numbers such that the following
hold for all $k$ in $\N$, for all $j=0,\ldots,d$ and for all $g$ in $G$:
\bi\item $\left\|\alpha_g\left((\widetilde{\theta}_j^{(k)}(n_k))(x)\right)-\left(\widetilde{\theta}_j^{(k)}(n_k)\right)(x)\right\|<\frac{1}{k}$ for all $x$ in $M_r$
with $\|x\|\leq 1$.
\item $\left\|\alpha_g\left((\widetilde{\eta}_j^{(k)}(n_k))(y)\right)-\left(\widetilde{\eta}_j^{(k)}(n_k)\right)(y)\right\|<\frac{1}{k}$ for all $y$ in $M_{r+1}$
with $\|y\|\leq 1$.
\item $\left\|\sum\limits_{j=0}^d \left(\widetilde{\theta}_j^{(k)}(n_k)\right)(1)+\left(\widetilde{\eta}_j^{(k)}(n_k)\right)(1)-1\right\|<\frac{1}{k}$
\ei
With $\kappa_A\colon \ell^\I_\alpha(\N,A)\to A_{\I,\alpha}$ denoting the quotient map, it follows that for $j=0,\ldots,d$,
the maps
$$\theta_j=\kappa_A\circ \left(\widetilde{\theta}^{(1)}_j(n_1),\widetilde{\theta}_j^{(2)}(n_2),\ldots\right)\colon M_r\to (A_{\I}\cap A')^{\alpha_\I}$$
and
$$\eta_j=\kappa_A\circ \left(\widetilde{\eta}^{(1)}_j(n_1),\widetilde{\eta}_j^{(2)}(n_2),\ldots\right)\colon M_{r+1}\to (A_{\I}\cap A')^{\alpha_\I}$$
are completely positive contractive order zero, and satisfy $\sum\limits_{j=0}^d\theta_j(1)+\eta_j(1)=1$. This proves the claim,
and finishes the proof of the lemma.
\end{proof}

We now need to introduce a certain averaging technique that will allow us to take averages over the group in such a way that
$\ast$-algebraic relations are approximately preserved. A simplified version of this technique already appeared in \cite{gardella Kirchberg 1}
for circle actions with the Rokhlin property. \\
\ \\
\indent Let $G$ be a compact group, let $A$ be a unital \ca, and let $\alpha\colon G\to\Aut(A)$ be a continuous action.
Identify $C(G)\otimes A$ with $C(G,A)$, and denote by $\gamma\colon G\to\Aut(C(G,A))$ the diagonal action, this
is, $\gamma_g(a)(h)=\alpha_g(a(g^{-1}h))$ for all $g,h\in G$ and all $a\in C(G,A)$. Define an \emph{averaging
process} $\phi\colon C(G,A)\to C(G,A)$ by
$$\phi(a)(g)=\alpha_g(a(1))$$
for all $a$ in $C(G,A)$ and all $g$ in $G$. \\
\ \\
\indent For use in the proof of the following lemma, we recall the following standard fact about self-adjoint
elements: if $A$ is a \uca\, and $a,b\in A$ with $b^*=b$, then $-\|b\|a^*a\leq a^*ba\leq \|b\|a^*a$.

\begin{lemma} \label{average}
Let $A$ be a unital \ca, let $G$ be a compact group and let $\alpha\colon G\to\Aut(A)$ be a continuous action.
Denote by $\phi\colon C(G,A)\to C(G,A)$ the averaging process defined above, and by $\gamma\colon G\to\Aut(C(G,A))$
the diagonal action. Given $\ep>0$ and given a compact
set $F\subseteq C(G,A)$, there exist a positive number $\delta>0$, a finite subset $K\subseteq G$ and continuous
functions $f_k$ in $C(G)$ for $k$ in $K$ such that
\be
\item If $g$ and $h$ in $G$ satisfy $d(g,h)<\delta$, then $\|\gamma_{g}(a)-\gamma_{h}(a)\|<\ep$
for all $a$ in $\bigcup\limits_{g\in G}\gamma_{g}(F)$
\item We have $0\leq f_k\leq 1$ for all $k$ in $K$.
\item The family $(f_k)_{k\in K}$ is a partition of unity for $G$.
\item For $k_1$ and $k_2$ in $K$, whenever $f_{k_1}f_{k_2}\neq 0$, then $d(k_1,k_2)<\dt$.
\item For every $g \in G$ and every $a$ in $\bigcup\limits_{g\in G}\gamma_{g}(F)$, we have
$$\left\| \phi(a)(g)-\sum\limits_{k\in K} f_k(g)\alpha_{k}(a(1))\right\|<\varepsilon.$$\ee
\end{lemma}
\begin{proof}
We claim that the averaging process $\phi\colon C(G,A)\to C(G,A)$
is a homomorphism. Let $a,b\in C(G,A)$, and let $g$ in $G$. We have
\begin{align*} (\phi(a)\phi(b))(g)=\alpha_g(a(1))\alpha_g(b(1))=\alpha_g(ab(1))=\phi(ab)(g),\end{align*}
showing that $\phi$ is multiplicative. It is clearly linear and preserves the involution, so it is a homomorphism.\\
\indent We claim that $\gamma_g(\phi(a))=\phi(a)$ for all $g$ in $G$ and all $a$ in $C(G,A)$. Indeed,
for $h$ in $G$, we have
\begin{align*}
\gamma_g(\phi(a))(h)&=\alpha_g(\phi(a)(g^{-1}h))=\alpha_g(\alpha_{g^{-1}h}(a(1)))\\
&=\alpha_h(a(1))=\phi(a)(h),
\end{align*}
which proves the claim.\\
\indent Set $F'=\bigcup\limits_{g\in G}\gamma_{g}(F)$, which is a compact subset of $C(G,A)$. Since every element in a \ca\ is the linear
combination of two self-adjoint elements, we may assume without loss of generality that every element of $F'$ is self-adjoint.
Set
$$F''=\left\{a(g)\colon a\in F',g\in G\right\},$$
which is a compact subset of $A$. Using continuity of $\alpha$, choose $\delta>0$ such that whenever $g$ and $h$
in $G$ satisfy $d(g,h)<\delta$, then
$\|\alpha_{g}(a)-\alpha_{h}(a)\|<\ep$
for all $a$ in $F''$. Given $g$ in $G$, denote by $U_g$ the open ball centered at $g$ with radius $\frac{\delta}{2}$. Let $K\subseteq G$
be a finite subset such $\bigcup\limits_{k\in K}U_k=G$, and let $(f_k)_{k\in K}$ be a partition of unity subordinate to $\{U_k\}_{k\in K}$.
Given $g$ in $G$ and $a$ in $F'$, we have
\begin{align*} \phi(a)(g)-\sum\limits_{k\in K} f_k(g)\alpha_{k}(a(1))
&=\sum\limits_{k\in K} f_k(g)^{1/2}(\alpha_{k}(a(1))-\alpha_g(a(1)))f_k(g)^{1/2}\\
&\leq \sum\limits_{j=1}^{n} \|\alpha_{k}(a(1))-\alpha_g(a(1))\|f_k(g).
\end{align*}
Now, for $k\in K$, if $f_k(g)\neq 0$, then $d(g,k)<\delta$, and hence $\|\alpha_{k}(a(1))-\alpha_g(a(1))\|<\ep$.
In particular, we conclude that
$$-\ep< \phi(a)(g)-\sum\limits_{k\in K} f_k(g)\alpha_{k}(a(1))<\varepsilon.$$
This shows that condition (5) in the statement is satisfied, and finishes the proof.
\end{proof}

Let $A$ be a \ca, let $G$ be a compact group, and let $\alpha\colon G\to\Aut(A)$ be a continuous action.
We denote by $E\colon A\to A^\alpha$ the standard
conditional expectation. If $\mu$ denotes the normalized Haar measure on $G$, then $E$ is given by
$$E(a)=\int_G \alpha_g(a)\ d\mu(g)$$
for all $a$ in $A$.

\begin{prop}\label{prop: partition of unity and Rp}
Let $A$ be a \uca, let $G$ be a compact group, let $d$ be a non-negative integer, and let $\alpha\colon G\to\Aut(A)$ be an action
with $\dimRok(\alpha)\leq d$. For every $\varepsilon>0$ and every be a compact subset $F$ of $A$, there exist $\delta>0$, a
finite subset $K\subseteq G$, continuous functions $f_{k}$ in $C(G)$ for $k$ in $K$, and completely
positive contractive linear maps $\psi_0,\ldots,\psi_d\colon C(G)\to A$ such that
\be\item If $g$ and $g'$ in $G$ satisfy $d(g,g')<\delta$, then $\|\alpha_g(a)-\alpha_{g'}(a)\|<\ep$
for all $a$ in $F$.
\item We have $0\leq f_{k}\leq 1$ for all $k$ in $K$.
\item Whenever $k$ and $k'$ in $K$ satisfy $f_kf_{k'}\neq 0$, then $d(k,k')<\delta$.
\item For every $g\in G$, for every $j=0,\ldots,d$, and for every $a\in F$, we have
$$\left \| \ \alpha_g\left(\sum\limits_{k\in K} \psi_j(f_{k})^{1/2}\alpha_{k}(a)\psi_j(f_{k})^{1/2}\right)-
\sum\limits_{k\in K} \psi_j(f_{k})^{1/2}\alpha_{k}(a)\psi_j(f_{k})^{1/2}\right\|<\varepsilon.$$
\item For every $a\in F$, for all $k\in K$, and for every $j=0,\ldots,d$, we have
$$\left\|a\psi_j(f_{k})-\psi_j(f_{k})a\right\|<\frac{\ep}{|K|} \ \ \mbox{ and } \ \
\left\|a\psi_j(f_{k})^{1/2}-\psi_j(f_{k})^{1/2}a\right\|<\frac{\ep}{|K|}.$$
\item Whenever $k$ and $k'$ in $K$ satisfy $f_{k}f_{k'}=0$, then for all $j=0,\ldots,d$ we have
$$\left\|\psi_j(f_k)^{1/2}\psi_j(f_{k'})^{1/2}\right\|<\frac{\ep}{|K|}.$$
\item The family $(f_{k})_{k\in K}$ is a partition of unity for $G$, and moreover,
$$\left\|\sum\limits_{j=0}^d\sum\limits_{k\in K} \psi_j(f_k)-1\right\|<\frac{\ep}{|K|}.$$
\ee
Moreover, if $\cdimRok(\alpha)\leq d$, then the choices above can be made so that
in addition to conditions (1) through (7) above, we have:
\be\setcounter{enumi}{7}
\item For all $j,\ell=0,\ldots,d$ and for all $k,k'$ in $K$,
$$\left\|\left[\psi_j(f_k),\psi_\ell(f_{k'})\right]\right\|<\ep \ \ \mbox{ and } \ \
\left\|\left[\psi_j(f_k)^{1/2},\psi_\ell(f_{k'})^{1/2}\right]\right\|<\frac{\ep}{|K|}.$$
\ee
\end{prop}
\begin{proof} Without loss of generality, we may assume that $F$ is $\alpha$-invariant.
Using Lemma \ref{average}, choose a positive number $\delta>0$, a finite subset $K\subseteq G$, and continuous
functions $f_k$ in $C(G)$ for $k$ in $K$, such that conditions (1) through (5) in
Lemma \ref{average} are satisfied for $F$ and $\frac{\ep}{2}$. Set
$S=\left\{f_k,f_k^{1/2}\colon k\in K\right\}\subseteq C(G),$
and for every $m$ in $\N$, choose completely positive contractive maps
$$\psi^{(m)}_0,\ldots,\psi^{(m)}_d\colon C(G)\to A$$
as in the conclusion of part (1) in Lemma 3.7 in
\cite{gardella Rdim} for the choices of finite set $S\subseteq C(G)$, compact subset $F\subseteq A$, and tolerance
$\frac{1}{m}$. Identify $C(G,A)$ with $C(G)\otimes A$, and for $m$ in $\N$ and $j=0,\ldots,d$ define a completely positive contractive map
$\phi^{(m)}_j\colon C(G,A)\to A$ by $\phi^{(m)}_j=\psi^{(m)}_j\otimes \id_A$. It is clear that $\phi^{(m)}_j$ is equivariant, where we take
$C(G,A)$ to have the diagonal action $\gamma$ of $G$. \\
\indent Given $a$ in $F$, given $j=0,\ldots,d$, and given $g$ in $G$, we have
the following, where use condition (5) in the conclusion of Lemma \ref{average} at the last step:
\begin{align*}
\limsup_{m\to\I}&\left\| \alpha_g\left(\sum\limits_{k\in K} \psi^{(m)}_j(f_k)^{1/2} \alpha_{k}(a)\psi^{(m)}_j(f_k)^{1/2}\right) \right.\\
& \ \ \ \ \ \ \ \ \left. - \sum\limits_{k\in K} \psi^{(m)}_j(f_k)^{1/2} \alpha_{k}(a)\psi^{(m)}_j(f_k)^{1/2}\right\| \\
& \ = \limsup_{m\to\I}\left\|\alpha_g\left( \phi^{(m)}_j\left(\sum\limits_{k\in K}f_k\otimes \alpha_{k}(a)\right)\right)-
\phi^{(m)}_j\left(\sum\limits_{k\in K}f_k\otimes \alpha_{k}(a)\right)\right\| \\
& \ = \limsup_{m\to\I}\left\|\phi_j^{(m)}\left(\gamma_g\left(\sum\limits_{k\in K}f_k\otimes \alpha_{k}(a)\right)-
\sum\limits_{k\in K}f_k\otimes \alpha_{k}(a)\right)\right\|\\
& \ \leq \limsup_{m\to\I}\left\|\gamma_g\left(\sum\limits_{k\in K}f_k\otimes \alpha_{k}(a)\right)-
\sum\limits_{k\in K}f_k\otimes \alpha_{k}(a)\right\|\leq \frac{\ep}{2}
\end{align*}
The result in the case that $\dimRok(\alpha)\leq d$ follows by choosing $m>\frac{|K|}{\ep}$ large enough so that
$$\left\| \alpha_g\left(\sum\limits_{k\in K} \psi^{(m)}_j(f_k)^{1/2} \alpha_{k}(a)\psi^{(m)}_j(f_k)^{1/2}\right)-
\sum\limits_{k\in K} \psi^{(m)}_j(f_k)^{1/2} \alpha_{k}(a)\psi^{(m)}_j(f_k)^{1/2}\right\|<\ep.$$

If one moreover has $\cdimRok(\alpha)\leq d$, one uses part (2) of Lemma 3.7 in \cite{gardella Rdim}, and the same
argument shows that the choices can be made so that condition (8) in this proposition is also satisfied. We omit the details.
\end{proof}

We are now ready to prove that absorption of the Jiang-Su algebra $\mathcal{Z}$ passes to crossed products and fixed point algebras
by compact group actions with finite Rokhlin dimension with commuting towers. This generalizes Theorem 5.9 in \cite{HWZ}, and it partially generalizes part (1) of Corollary 3.2 in \cite{hirshberg winter}. \\
\indent We do not know whether commuting towers are really necessary in the theorem below. In view of the Toms-Winter conjecture
and Theorem \ref{nuc dim}, this condition should not be necessary if both $A$ and $A\rtimes_\alpha G$ are simple.

\begin{thm} \label{thm: cp Z-abs}
Let $A$ be a separable \uca, let $G$ be a compact group, and let $\alpha\colon G\to\Aut(A)$ be
a continuous action with finite Rokhlin dimension with commuting towers. Suppose that $A$ is $\mathcal{Z}$-absorbing.
Then the crossed product $A\rtimes_\alpha G$ and the fixed point algebra $A^\alpha$ are also
$\mathcal{Z}$-absorbing.\end{thm}
\begin{proof} We show first that the crossed product $A\rtimes_\alpha G$ is $\mathcal{Z}$-absorbing.
Our proof combines the methods of Theorem 5.9 in \cite{HWZ} and, to a lesser extent, Theorem 3.3 in
\cite{hirshberg winter}. We will produce maps as in the statement of Lemma \ref{lma: Z-abs}. \\
\indent Let $d=\text{dim}_{\text{Rok}}^c(\alpha)$. Fix a positive
integer $r$ in $\N$ and a compact subset $F\subseteq A$, which, without loss of generality, we assume to be
$\alpha$-invariant. We may also assume that $F$ contains only self-adjoint elements of norm at most 1.
Choose order zero maps $\theta\colon M_r\to \mathcal{Z}$
and $\eta\colon M_{r+1}\to\mathcal{Z}$ with commuting ranges satisfying $\theta(1)+\eta(1)=1$. Define unital
homomorphisms
$$\iota_0,\ldots,\iota_d\colon \mathcal{Z}\to A \hookrightarrow A_{\I,\alpha}$$
as follows. Start with any unital homomorphism $\iota_0\colon \mathcal{Z}\to A$ satisfying
$$\|\iota_0(z)a-a\iota_0(z)\|<\ep$$
for all $z$ in $\mathcal{Z}$ and all $a$ in $F$, which exists because $A$ is $\mathcal{Z}$-absorbing.
Once we have constructed $\iota_j$ for $j=0,\ldots,k-1$, we choose
$\iota_k\colon\mathcal{Z}\to A$ such that
$$\|\iota_k(z)b-b\iota_k(z)\|<\ep$$
for all $z$ in $\mathcal{Z}$ and for all $b$ in the compact $\alpha$-invariant set
$$F\cup \bigcup\limits_{j=0}^{k-1} \bigcup\limits_{g\in G}\alpha_g\left(\left\{(\iota_j\circ\theta)(x), (\iota_j\circ\eta)(y)\colon x \in M_r,y\in M_{r+1}, \|x\|,\|y\|\leq 1\right\}\right).$$
\indent Set
$$F'=F\cup \bigcup\limits_{j=0}^d \bigcup\limits_{g\in G}\alpha_g\left(\left\{(\iota_j\circ\theta)(x), (\iota_j\circ\eta)(y)\colon x \in M_r,y\in M_{r+1}, \|x\|,\|y\|\leq 1\right\}\right),$$
which we regard as a subset of $A_{\I,\alpha}$ via the inclusion $A\hookrightarrow A_{\I,\alpha}$.
Choose a finite subset $K\subseteq G$, continuous functions $f_k$ in $C(G)$ for
$k$ in $K$, and unital completely positive contractive maps $\varphi_0,\ldots,\varphi_d\colon C(G)\to A\hookrightarrow A_{\I,\alpha}\cap A'$
as in the conclusion of Proposition \ref{prop: partition of unity and Rp} for the choices of compact set $F'\subseteq A_{\I,\alpha}$ and
tolerance $\ep$. Then the ranges of $\varphi_0,\ldots,\varphi_d$ commute with the ranges of the homomorphisms $\iota_0,\ldots,\iota_d$. \\
\indent For $j=0,\ldots,d$, define
$$\theta_j\colon M_r\to A_\I\cap A' \ \ \mbox{ and } \ \ \eta_j\colon M_{r+1}\to A_\I\cap A'$$
by
\[\theta_j(x)=\sum\limits_{k\in K}\varphi_j(f_k)\alpha_k((\iota_j\circ\theta)(x)) \ \ \mbox{ and } \ \
\eta_j(y)=\sum\limits_{k\in K}\varphi_j(f_k)\alpha_k((\iota_j\circ\eta)(y))\]
for all $x$ in $M_r$ and all $y$ in $M_{r+1}$. We claim that these maps satisfy the conditions in the statement
of Lemma \ref{lma: Z-abs}.\\
\ \\
\indent Condition (1). Let $j,\ell \in \{0,\ldots,d\}$, let $x$ in $M_r$ and let $y\in M_{r+1}$ satisfy $\|x\|,\|y\|\leq 1$. Without
loss of generality, we may assume that $x^*=x$ and $y^*=-y$. Then $\left[\theta_j(x),\eta_\ell(y)\right]$ is self-adjoint and
\begin{align*}
\left[\theta_j(x),\eta_\ell(y)\right]&= \sum\limits_{k,k'\in K} \left[\varphi_j(f_k)\alpha_k((\iota_j\circ\theta)(x)),\varphi_\ell(f_{k'})\alpha_{k'}((\iota_\ell\circ\eta)(y))\right]\\
& = \sum\limits_{k,k'\in K} \varphi_j(f_k)\varphi_\ell(f_{k'})\alpha_k\left(\left[(\iota_j\circ\theta)(x),\alpha_{k^{-1}k'}((\iota_\ell\circ\eta)(y))\right]\right)\\
&\leq \sum\limits_{k,k'\in K} \varphi_j(f_k)\varphi_\ell(f_{k'})\left\|\left[(\iota_j\circ\theta)(x),\alpha_{k^{-1}k'}((\iota_\ell\circ\eta)(y))\right]\right\|\\
&< \left(\sum\limits_{k,k'\in K} \varphi_j(f_k)\varphi_k(f_{k'})\right) \ep=\ep.
\end{align*}
Likewise, $\left[\theta_j(x),\eta_\ell(y)\right]>-\ep$, so $\|\left[\theta_j(x),\eta_\ell(y)\right]\|<\ep$, as desired.\\
\ \\
\indent Condition (2). Given $j=0,\ldots,d$ and given positive orthogonal elements $x,x'\in M_r$ with $\|x\|,\|x'\|\leq 1$, set
$a=(\iota_j\circ\theta)(x)$ and $b=(\iota_j\circ\theta)(x')$. Then $ab=0$ because $\theta$ is an order
zero map and $\iota_j$ is a homomorphism. Using that $f_kf_{k'}\neq 0$ implies $d(k,k')<\delta$ for $k,k'\in K$ at
the fourth step, we have
\begin{align*}
\left\|\theta_j(x)\theta_j(x')\right\|&=\left\|\sum\limits_{k,k'\in K} \varphi_j(f_k)\alpha_h(a)\varphi_j(f_{k'})\alpha_{k'}(b)\right\|\\
&= \left\|\sum\limits_{f_kf_{k'}\neq 0,h\neq h'} \varphi_j(f_k)\varphi_j(f_{k'})\alpha_k(a)\alpha_{k'}(b)\right\|\\
&< |K|\frac{\ep}{|K|} + \left\|\sum\limits_{f_kf_{k'}\neq 0, k\neq k'} \varphi_j(f_k)\varphi_j(f_{k'})\alpha_k(ab)\right\|=\ep,
\end{align*}
as desired. A similar computation shows that $\|\eta_j(y)\eta_j(y')\|<\ep$ whenever $y$ and $y'$
are as in Condition (2) of Lemma \ref{lma: Z-abs}.\\
\ \\
\indent Condition (3). Given $x$ in $M_r$ with $\|x\|\leq 1$, given $g$ in $G$, and given $j=0,\ldots,d$,
the same computation
carried out in the verification of condition (4) in Proposition \ref{prop: partition of unity and Rp} shows that
$\|(\alpha_\I)_g(\theta_j(a))-\theta_j(a)\|<\ep$, as desired. A similar computation shows that $\|(\alpha_\I)_g(\eta_j(y))-\eta_j(y)\|<\ep$
whenever $y$ is as in Condition (3) of Lemma \ref{lma: Z-abs}.\\
\ \\
\indent Condition (4). Let $a$ in $F$, let $j=0,\ldots,d$ and let $x$ in $M_r$ satisfy $\|x\|\leq 1$ and $x^*=-x$. Then
$[a,\theta_j(x)]$ is self-adjoint because $a^*=a$, and moreover we have
\begin{align*}
[a,\theta_j(x)]&= \sum\limits_{k\in K}\varphi_j(f_k)[a,\alpha_k((\iota_j\circ\theta)(x))]\\
&\leq \sum\limits_{k\in K}\varphi_j(f_k)\left\|[a,\alpha_k((\iota_j\circ\theta)(x))]\right\|\\
&< \ep\left( \sum\limits_{k\in K}\varphi_j(f_k)\right)=\ep.
\end{align*}
Likewise, $\left[a,\theta_j(x)\right]>-\ep$, so $\|\left[a,\theta_j(x)\right]\|<\ep$, as desired. A similar computation
shows that $\|[a,\eta_j(y)]\|<\ep$ whenever $y$ is as in Condition (4) of Lemma \ref{lma: Z-abs}.\\
\ \\
\indent Condition (5). We use that the family $(f_k)_{k\in K}$ is a partition of unity of $G$ in the third step to conclude
that
\begin{align*} \sum\limits_{j=0}^d\theta_j(1)+\eta_j(1)& =\sum\limits_{j=0}^d\sum\limits_{k\in K}\varphi_j(f_k)\alpha_k(\iota_j(\theta(1)+\eta(1)))\\
&=\sum\limits_{j=0}^d\sum\limits_{k\in K}\varphi_j(f_k)\\
&=\sum\limits_{j=0}^d\varphi_j(1)\\
&=1.\end{align*}
The result for $A\rtimes_\alpha G$ now follows from Lemma \ref{lma: Z-abs}.\\
\indent Since the fixed point algebra $A^\alpha$ is a corner in $A\rtimes_\alpha G$ by the Theorem in \cite{Rosenberg},
it follows from Corollary 3.1 in \cite{ssa} that $A^\alpha$ is also $\mathcal{Z}$-absorbing.  \end{proof}

\section{Open problems}

In this last section, we give some indication of possible directions for future work, and raise some natural questions related
to our findings. Some of these questions will be addressed in \cite{gardella hirshberg santiago}.\\
\indent Theorem \ref{thm: cp Z-abs} and Corollary 3.4 in \cite{hirshberg winter} suggest that the following conjecture may be true.

\begin{cnj} Let $G$ be a second-countable compact group, let $A$ be a separable unital \ca, and let $\alpha\colon G\to\Aut(A)$ be an action
 of $G$ on $A$ with finite Rokhlin dimension with commuting towers. Let $\mathcal{D}$ be a strongly self-absorbing \ca\ and
suppose that $A$ is $\mathcal{D}$-absorbing. Then $A\rtimes_\alpha G$ is $\mathcal{D}$-absorbing.
\end{cnj}

We point out that the corresponding result for noncommuting towers is not in general true: Example 4.8 in
\cite{gardella Rdim}
shows that $\mathcal{O}_2$-absorption is not preserved. It moreover shows that UHF-absorption is also not preserved: except for the
UHF-algebra $M_{2^\I}$, this is contained in said example, and to show that $M_{2^\I}$-absorption is not in general preserved either, one
may adapt the construction of Izumi to produce a
$\Z_3$ action on $\mathcal{O}_2$ with Rokhlin dimension 1 with noncommuting
towers such that the $K$-theory of the crossed product is not uniquely 2-divisible
(see, for example, \cite{gardella lupini}), ruling out $M_{2^\I}$-absorption as well.\\
\indent One should also explore preservation of other structural properties besides $\mathcal{D}$-absorption.

\begin{pbm} Can one generalize some of the parts of Theorem 2.6 in \cite{phillips survey} to compact (or finite) group actions
 with finite Rokhlin dimension with commuting towers?
\end{pbm}

As pointed out before, one should not expect much if only noncommuting towers are assumed (except for (1) and (8) -- without UCT,
which are true for arbitrary pointwise outer actions of discrete amenable groups). Also, it is probably easy to construct
counterexamples to several parts of Theorem 2.6 in \cite{phillips survey} even in the case of finite Rokhlin dimension with commuting
towers. \\
\indent We provide one such counterexample here.

\begin{prop} There exist a \uca\ $A$ and an action $\alpha\colon\Z_2\to\Aut(A)$ with $\cdimRok(\alpha)=1$, such that the map
 $K_\ast(A^\alpha)\to K_\ast(A)$ induced by the canonical inclusion $A^\alpha\to A$, is not injective. In particular,
Theorem 3.13 of \cite{izumi finite group actions I} (where simplicity of $A$ is not needed -- see \cite{gardella rokhlin property}) does not
hold for finite group actions with finite Rokhlin dimension with commuting towers.
\end{prop}
\begin{proof}
Denote by $\R\mathbb{P}^2$ the real projective plane, and set $X=\T\times\R\mathbb{P}^2$. Define an action $\alpha$ of $\Z_2$ on $X$ via $(\zeta,r)\mapsto (-\zeta,r)$
for all $(\zeta,r)\in X$. Then $\alpha$ is the restriction of the product action $\texttt{Lt}_{\T}\times \id_{\R\mathbb{P}^2}$ of $\T$ on $X$. Since this
product action has the Rokhlin property, it follows from Theorem 3.10 in \cite{gardella Rdim} that $\cdimRok(\alpha)\leq 1$. Since $X$ has no non-trivial projections, it must be
$\cdimRok(\alpha)=1$.\\
\indent One has $X/\Z_2\cong (\T/\Z_2)\times \R\mathbb{P}^2$, and the canonical quotient map $\pi\colon X\to X/\Z_2$ is given by $\pi(\zt,r)=(\zt^2,r)$
for $(\zt,r)\in X$. The induced map
$$K^1(\pi)\colon K^0(\R\mathbb{P}^2)\oplus K^1(\R\mathbb{P}^2)\to K^0(\R\mathbb{P}^2)\oplus K^1(\R\mathbb{P}^2)$$
is easily seen to be given by $K^1(\pi)(a,b)=(2a,b)$ for $(a,b)\in K^0(\R\mathbb{P}^2)\oplus K^1(\R\mathbb{P}^2)$. Since
$K^0(\R\mathbb{P}^2)\cong \Z\oplus \Z_2$, we conclude
that $K^1(\pi)$ is not injective, and hence neither is $K_1(\iota)$.\end{proof}

\end{document}